\documentclass{article}

\usepackage{amsthm}
\usepackage[english]{babel}
\usepackage{amsmath}
\usepackage{amssymb}
\usepackage[%
]{hyperref}
\usepackage{booktabs}
\usepackage{enumerate}
\usepackage[left = 25mm, top = 40mm, bottom = 35mm, right = 25mm]{geometry}
\usepackage{multicol}
\usepackage{bbm}
\usepackage[]{todonotes}
\usepackage{setspace}
\usepackage{mathtools}

\theoremstyle{definition}
\newtheorem{definition}{Definition}
\newtheorem{Remark}[definition]{Remark}
\newtheorem{Example}[definition]{Example}

\newtheorem*{Fact}{Fact}
\theoremstyle{plain}
\newtheorem{theorem}[definition]{Theorem} 
\newtheorem{theoremintro}{Theorem}
\newtheorem{lemma}[definition]{Lemma}

\newtheorem{Corollary}[definition]{Corollary}

\newcommand{\soc}{\operatorname{soc}}
\newcommand{\Z}{\mathbb{Z}}
\newcommand{\Ann}{\operatorname{Ann}}
\newcommand{\N}{\mathbb{N}}
\newcommand{\Aut}{\operatorname{Aut}}
\newcommand{\F}{\mathbb{F}}

\newcommand{\ord}{\operatorname{ord}}
\newcommand{\Syl}{\operatorname{Syl}}
\newcommand{\Cl}{\operatorname{Cl}}
\newcommand{\SL}{\operatorname{SL}}
\newcommand{\AGL}{\operatorname{AGL}}

\newcommand{\DCLab}{\mathcal{C}}

\newcommand{\SIDCLab}{\overline{\mathcal{C}}^+}

\numberwithin{definition}{section}
\numberwithin{equation}{section}

\title{Classifying group algebras in which the socle of the center is an ideal\\[0.2cm]}
\author{\textbf{Sofia Brenner}\footnote{Present address: Department of Mathematics, TU Darmstadt, Germany. E-mail address: \texttt{sofia.brenner@tu-darmstadt.de}} \\
	\normalsize\emph{Institute for Mathematics, Friedrich Schiller University Jena, Germany}\\
	\texttt{sofia.bettina.brenner@uni-jena.de}}
\date{\vspace{-0.5cm}}

\begin{document}
\maketitle
\begin{abstract}
\noindent
Let $F$ be a field. We study the structure of finite groups $G$ for which the socle of the center of the group algebra $FG$ is an ideal in $FG$. In \cite{BRE222}, we studied which general implications this ring-theoretic property of~$FG$ has on the structure of $G$. In the present article, we now focus on a detailed structural analysis of the groups with this property. We showcase its strength by deriving a complete classification for a prototype class of groups. 
%
%

\end{abstract}

\section{Introduction}

Let $F$ be a field and let $G$ be a finite group. A natural question to study is which aspects of the structure of~$G$ are determined by (ring-theoretic) properties of its group algebra $FG$. Featuring as Problem 16 in Brauer's famous list of problems~\cite{BRA63}, this question has seen extensive research. This is motivated by the fact that group algebras are the central objects studied in the representation theory of finite groups.  In this paper, we contribute to this line of research by studying the following question:
\smallskip
\begin{center}
\emph{For which finite groups $G$ is the socle $\soc(ZFG)$ of the center of $FG$ an ideal in $FG$?}
\smallskip
\end{center}

The corresponding problem was studied previously for the Jacobson radical of the center of $FG$ (see \cite{CLA69, KOS78, KUL20}) as well as for the Reynolds ideal (see \cite[Theorem~A]{BRE222}).
General results on symmetric algebras in which the socle of the center is an ideal were obtained in \cite{BRE221}. In~\cite{BRE222}, we focused on deriving general structural properties of a finite group $G$ from the assumption that $\soc(ZFG)$ is an ideal in $FG$. 
In contrast, the aim of the present paper is a precise classification of the groups with this property, which underlines the strength of its structural implications. Using the results of \cite{BRE222}, we first derive a detailed description of the groups with this property. For a prototype case, we then derive a complete classification. In essence, we show that these groups can be built as central products of extensions of Camina $p$-groups by an automorphism of large order.
\medskip

The starting point are the reductions of the main problem from \cite{BRE222}, which can be summarized as follows: 
%

\begin{Fact}[{\cite[Remark 3.1, Theorems 3 and 4]{BRE222}}]
We may assume that $F$ is an algebraically closed field of characteristic $p > 0$. Furthermore, it suffices to study finite groups $G$ with the following properties:
\begin{enumerate}[(i)]
	\item $G'$ is a Sylow $p$-subgroup of $G$ (in particular, we have $G \cong G' \rtimes H$ for an abelian $p'$-group~$H$), and 
	\item $O_{p'}(G) = 1$.  
\end{enumerate}
\end{Fact}

For the sake of simplicity, we call a group $G$ \emph{basic} if it satisfies (i) and (ii). The reduction of our main problem to basic groups will be reviewed in Section~\ref{sec:introduction}. 
%
%
\medskip

Let $G \cong G' \rtimes H$ be a finite basic group such that $\soc(ZFG)$ is an ideal in $FG$. By \cite[Proposition~3.20]{BRE222}, we have \[G'' \subseteq \Phi(G') \subseteq Z(G').\] Set $D \coloneqq G/G''$ and $Z_D \coloneqq Z(G')/G''$. Note that $D \cong D' \rtimes H$, where we identify $H$ with its image in $D$. As a central step, we obtain the following decomposition result:
\begin{theoremintro}\label{theo:uniquesolutiond}
	Let $F$ be a field of characteristic $p > 0$, and let $G$ be a finite basic group for which $\soc(ZFG)$ is an ideal in $FG$. The group $D$ has the following properties:
	\begin{enumerate}[(i)]
		\item $D' = T \times Z_D$ for an elementary abelian normal subgroup $T$ of $D$.
		\label{D1}
		\item Write $T = T_1 \times \dots \times T_n$ for minimal normal subgroups $T_1, \ldots, T_n$ of $D$. Then
		\[H/C_H(T) \cong H/C_H(T_1) \times \dots \times H/C_H(T_n),\]
		where, for $i \in \{1,\ldots, n\}$, the group $H/C_H(T_i)$ is cyclic of order $|T_i|-1$. 
		\label{D2}
		\label{D3}
	\end{enumerate} 
\end{theoremintro}

In contrast, the analysis of the action of $H$ on $Z_D$ is more complicated. For this reason, we focus on the case $Z_D = 1$, that is, $Z(G') = G''$, which we believe to be a cornerstone of the general classification. Our main result is the following characterization:  


\begin{theoremintro}\label{theo:b}
	Let $F$ be a field of characteristic $p > 0$ and let $G$ be a finite basic group satisfying $Z(G') = G''$. Then the following are equivalent: 
	\begin{enumerate}[(i)]
	\item $\soc(ZFG)$ is an ideal in $FG$.
	\item $G $ is a central product of groups $G_1, \ldots, G_n$ of the form $G_i = K_i \rtimes \langle \varphi_i \rangle$ ($i = 1, \dots, n$), where 
	\begin{enumerate}[(a)]
			\item $K_i$ is a Camina $p$-group of nilpotency class 2,
			\item $\varphi_i$ is an automorphism of order $|K_i : K_i'|-1$ of $K_i$ such that $\langle \varphi_i \rangle$ permutes the nontrivial cosets of~$K_i'$ in~$K_i$ transitively, and
			\item the action of $\langle \varphi_i \rangle$ on $K_i'$ is not faithful.
		\end{enumerate}
\end{enumerate}
\end{theoremintro}

In Example~\ref{ex:extraspecial}, we present families of groups naturally satisfying the conditions in Theorem~\ref{theo:b}\,(ii). The study of Camina $p$-groups with an automorphism of large order as in Theorem~\ref{theo:b}\,(ii) is interesting in its own right. 
\medskip

This paper is organized in the following way: In Section \ref{sec:introduction}, we introduce our notation and recall several results from \cite{BRE222}. In particular, we review the reduction from the general problem to basic groups. In Section~\ref{sec:nec}, we study the structure of the quotient group $G/G''$, proving Theorem~\ref{theo:uniquesolutiond}. Section~\ref{sec:generalgroupsfirstcase} is devoted to the classification in the special case $Z(G') = G''$, that is, to the proof of Theorem~\ref{theo:b}.

\section{Preliminaries}\label{sec:introduction}
In this section, we introduce the notation used in this paper and recall some results obtained in \cite{BRE222}. In particular, we explain how our main problem can be reduced to the investigation of basic groups (see Section~\ref{sec:reduction}).
\subsection{Notation}
In this paper, we always assume that $G$ is a finite group. As usual, let $G'$, $Z(G)$ and $\Phi(G)$ denote the derived subgroup, the center and the Frattini subgroup of $G$, respectively. We set $[a,b] = aba^{-1}b^{-1}$ for $a,b \in G$ and $[A,B] \coloneqq \langle [a,b] \colon a \in A,\ b \in B\rangle$ for $A,B \subseteq G$. Recall that $G = G_1 * G_2$ is the central product of subgroups $G_1$ and $G_2$ if we have $G = \langle G_1, G_2 \rangle$ and $[G_1, G_2] = 1$. For subsets $S$ and $T$ of $G$, set $C_T(S)$ to be the centralizer of $S$ in $T$. For $g \in G$, we write $[g]_G$ for the conjugacy class of $g$ in $G$. Note that $[g]_G = U_{G,g} g$ holds for $U_{G,g} \coloneqq \{[a,g] \colon a \in G\} \subseteq G'$. Let $\Cl(G)$ be the set of conjugacy classes in $G$ and write $g \sim h$ if $g,h \in G$ are conjugate in $G$. The group $G$ is called a Camina group if $[g]_G = g G'$ holds for all $g \in G \setminus G'$ (for results on the structure of Camina groups see, for example, \cite{DAR96}). For a prime $p \in \N$, the $p$-part and the $p'$-part of $g \in G$ are denoted by $g_p$ and $g_{p'}$, respectively. As usual, we set $O_p(G)$, $O_{p'}(G)$, $O^p(G)$ and $O^{p'}(G)$ to be the $p$-core, the $p'$-core, the $p$-residual subgroup and the $p'$-residual subgroup of $G$, respectively. Let $\Syl_p(G)$ be the set of Sylow $p$-subgroups of $G$. For $d \in \N$, let $\AGL(1,p^d) \cong \F_{p^d} \rtimes \F_{p^d}^\times$ denote the 1-dimensional affine linear group over $\F_{p^d}$. 
\medskip

For a finite-dimensional algebra $A$ over a field $F$, we consider its Jacobson radical $J(A)$ and its (left) socle~$\soc(A)$, the sum of all minimal left ideals of $A$. Both $J(A)$ and $\soc(A)$ are known to be ideals in $A$. In this paper, an ideal of $A$ is always meant to be a two-sided ideal, and $I \trianglelefteq A$ means that $I$ is an ideal of $A$.
\medskip

In the following, we study the group algebra $FG$ of a finite group $G$ over a field $F$. Recall that $FG$ is a symmetric algebra. For $X \subseteq G$, we set $$X^+ \coloneqq \sum_{x \in X} x \in FG.$$ It is well-known that the conjugacy class sums form an $F$-basis for the center $ZFG$ of $FG$. In this paper, we study the socle $\soc(ZFG)$ of $ZFG$, which is an ideal in $ZFG$, but not necessarily in $FG$. More precisely, $\soc(ZFG)$ is an ideal in $FG$ if and only if it is closed under multiplication with elements of $FG$. For a subset $S \subseteq ZFG$, let $\Ann_{ZFG}(S)$ be the annihilator of $S$ in $ZFG$. As $ZFG$ is finite-dimensional, we have $\soc(M) = \{m \in M \colon J(ZFG) m = 0\}$ for every $ZFG$-module $M$ (see \cite[Corollary~15.21]{AND74}). In particular, $\soc(ZFG) = \Ann_{ZFG}(J(ZFG))$.

\subsection{Previous work}

In this section, we summarize the results from \cite{BRE222} concerning the structure of the finite groups~$G$ for which $\soc(ZFG)$ is an ideal in $FG$. In Section~\ref{sec:reduction}, we explain several reduction steps. Section~\ref{sec:furtherwork} contains a summary of further results of \cite{BRE222} that are needed in the following. 

\subsubsection{Reductions}\label{sec:reduction}

In this section, we explain some reductions of the main problem obtained in \cite{BRE222}. In particular, we summarize the argument that it is sufficient to consider basic groups. 
\medskip

First, we consider the choice of the underlying field. As the problem is trivial for semisimple group algebras (see \cite[Remark 3.1]{BRE222}), it suffices to consider fields $F$ of characteristic $p > 0$. For every finite group $G$ and every field $F$ of characteristic $p > 0$, $\soc(Z\F_p G) \trianglelefteq \F_p G$ if and only if $\soc(ZFG) \trianglelefteq FG$. In the following, we can thus assume that~$F$ is algebraically closed. 
\medskip

As stated in the introduction, we call a finite group $G$ \emph{basic} if $G'$ is a Sylow $p$-subgroup of $G$ and $O_{p'}(G) = 1$ holds. Now let $G$ be an arbitrary finite group for which $\soc(ZFG)$ is an ideal in $FG$. Then $G = P \rtimes H$ for a Sylow $p$-subgroup~$P$ and an abelian $p'$-group $H$ (see \cite[Corollary~3.5]{BRE222}). For groups of this form, $\soc(ZFG) \trianglelefteq FG$ is equivalent to $\soc(ZF[G/O_{p'}(G)]) \trianglelefteq F[G/O_{p'}(G)]$ (see \cite[Lemma~3.10]{BRE222}). By going over to $G/O_{p'}(G)$, we may therefore assume $O_{p'}(G) = 1$. We refer to \cite[Section~3.3]{BRE222} for further details as well as various implications of these results in block theory. 
\medskip

By \cite[Theorem~4]{BRE222}, we have the central product decomposition $G = C_P(H) *O^p(G)$. Moreover, $\soc(ZFC_P(H))$ and $\soc(ZFO^p(G))$ are ideals in $FC_P(H)$ and $FO^p(G)$, respectively. Since the structure of the $p$-group $C_P(H)$ is determined by \cite[Theorem 2]{BRE222}, it suffices to consider the group $O^p(G)$. We may therefore assume $O^p(G) = G$, which implies that $G'$ is a Sylow $p$-subgroup of $G$ (see \cite[Remark 5.1]{BRE222}). Summarizing, we may assume that $G$ is a basic group. For further details on this reduction, we refer to \cite[Section~5]{BRE222}.

\subsubsection{Further results}\label{sec:furtherwork}
Let $F$ be an algebraically closed field of characteristic $p > 0$ and let $G$ be a finite group. In this section, we summarize results from \cite{BRE222} that are needed in the following sections. 
We frequently make use of the following criterion:
\begin{lemma}[{\cite[Lemma 3.3]{BRE222}}]\label{lemma:socideal}
	The socle $\soc(ZFG)$ is an ideal in $FG$ if and only if $\soc(ZFG) \subseteq (G')^+ \cdot FG$. 
\end{lemma} 

For central products, which arise frequently in our investigation, the following reduction theorem holds:

\begin{lemma}[{\cite[Theorem~3.33]{BRE222}}]\label{lemma:centralproduct}
	Let $G$ be the central product of subgroups $G_1$ and $G_2$. Then $\soc(ZFG)$ is an ideal in $FG$ if and only if $\soc(ZFG_i)$ is an ideal in $FG_i$ for each $i \in \{1,2\}$. 
\end{lemma}

By the previous section, we may assume 
that $G$ is a basic group. Until the end of this section, we additionally assume that $\soc(ZFG)$ is an ideal in $FG$. Note that
\[
\soc(FG) = \Ann_{FG}(J(FG)) = \Ann_{FG}(J(FG') FG) = (G')^+ \cdot FG,
\] as $G'$ is a Sylow $p$-subgroup of $G$ (see \cite[Theorem~1.11.10]{LIN18}).  
By Lemma~\ref{lemma:socideal}, we then obtain 
\[\soc(ZFG) = (G')^+ \cdot FG = \soc(FG).\]


Throughout, we use the following central result on the nilpotency class of $G'$:
\begin{lemma}[{\cite[Proposition~3.20]{BRE222}}]\label{lemma:nilpotencyclassofdg}
We have $\Phi(G') \subseteq Z(G')$. In particular, $G'$ has nilpotency class at most~2, and $G''$ is elementary abelian.
\end{lemma}

A substantial part of this paper consists of the investigation of the quotient group $D \coloneqq G/G''$. For $g \in G$, we set $\bar{g} \coloneqq gG'' \in D$ (similarly for subsets of $G$). Since $D \cong D' \rtimes \bar{H}$ with $\bar{H} \cong H$ holds, we usually identify~$\bar{H}$ and~$H$. The following observations will be used frequently throughout this paper:

\begin{Remark}\label{rem:centralizerhdd}\label{rem:trivia}
$\null$
	\begin{enumerate}[(i)]
		\item By \cite[Theorem 6.4.3]{GOR68}, we have $C_G(G') = Z(G')$. By \cite[Remark~5.1]{BRE222}, we have $G' = [G',H]$ and $C_{G'}(H) = Z(G)$. In particular, this yields $G = H G' = H [G',H]$. 
		\item By (i), we obtain $D' = [D',H]$. Hence $C_{D'}(H) = 1$ by \cite[Theorem~5.2.3]{GOR68}.
		\item Let $h \in C_H(D')$. Due to $G'' \subseteq \Phi(G')$, $h$ acts trivially on $G'/\Phi(G')$ and hence on~$G'$ by \cite[Theorem~5.1.1]{GOR68}. Since $H$ is abelian, this implies $h \in Z(G) \cap H \subseteq O_{p'}(G) = 1$. In particular, this yields $O_{p'}(D) \subseteq C_H(D') = 1$, so $D$ is basic.
		\item Let $h \in H$ and $a \in [h]_G$. We write $a = c h$ with $c \in  G'$. If $h$ and $c$ commute, we have $c = a_p$. Since $h$ and~$a$ are conjugate in $G$, this yields $c = 1$.
	\end{enumerate}
\end{Remark}

Throughout, we use the following basis for $J(ZFG)$: 

\begin{lemma}[{\cite[Theorem~3.14]{BRE222}}]\label{lemma:structjzfg}
For $C \in \Cl(G) \setminus \{1\}$, we set $b_C = C^+ - |C| \cdot 1$. 
	Then $\{b_C \colon C \in \Cl(G) \setminus \{1\}\}$ is an $F$-basis for $J(ZFG)$.
\end{lemma}

\begin{Remark}\label{rem:pdividesc}
Let $C \in \Cl(G)$. Then $p$ divides $|C|$ if and only if $C \not \subseteq Z(G')$ holds. 
\end{Remark}

%
%

In the following, we consider the canonical projection
$$\nu \colon FG \to FD,\ \sum_{g \in G} a_g g \mapsto \sum_{g \in G} a_g \bar{g}.$$

We set $\DCLab \coloneqq \{C \in \Cl(G) \setminus \{1\}\colon \nu(b_C) \neq 0\}$. By \cite[Lemma~3.28]{BRE222}, we have
\begin{equation}\label{eq:dcl}
\DCLab = \left\{C \in \Cl(G) \colon C \not \subseteq G'' \text{ and } p \text{ does not divide } \frac{|C|}{|\bar{C}|}\right\}.
\end{equation}
The subset $\SIDCLab \coloneqq \{b_{\bar{C}} \colon C \in \DCLab\}$ of $J(ZFD)$ plays a fundamental role in our investigation. Here, $b_{\bar{C}}$ denotes the basis element of $J(ZFD)$ corresponding to $\bar{C} \in \Cl(D)$ (see Lemma~\ref{lemma:structjzfg}). We mainly use the following special case of \cite[Theorem 3.31]{BRE222}:

\begin{theorem}\label{theo:anndecconjconstantcoeffs}
We have $\Ann_{ZFD}\bigl(\SIDCLab\bigr) = (D')^+ \cdot FD$.
\end{theorem} 

For $h \in H$, recall that $U_{D,h} = \{[d,h] \colon d \in D'\}$. 
When applying Theorem~\ref{theo:anndecconjconstantcoeffs}, we use the following description of $\Ann_{ZFD}(\SIDCLab)$:

\begin{lemma}\label{lemma:pstrichkonjugation}
We have 
\[\Ann_{ZFD}(\SIDCLab) =\bigcap_{\substack{C \subseteq Z(G'), \\ C \not \subseteq G''}} \Ann_{ZFD}(\bar{C}^+ - |\bar{C}| \cdot 1) \cap \bigcap_{\substack{\text{$p'$-conjugacy classes} \\ C \in \DCLab}} \Ann_{ZFD}(\bar{C}^+).\]
Moreover, if $h \in H$ and $C \coloneqq [h]_G \in \DCLab$, then $U_{D,h}$ is a nontrivial normal subgroup of $D$ and \[\Ann_{ZFD}(\bar{C}^+) = \Ann_{ZFD}(U_{D,h}^+).\]
\end{lemma}

\begin{proof}
First let $C \in \Cl(G)$ with $C \not \subseteq Z(G')$ and let $C' \in \Cl(G)$ be the conjugacy class containing the $p'$-parts of the elements in $C$. We claim that $C \in \DCLab$ implies $C' \in \DCLab$. To see this, let $g \in C$. As $G$ is solvable and $H$ is a Hall $p'$-subgroup of $G$, $g_{p'}$ is conjugate in $G$ to an element in $H$. 
Hence we may assume $h \coloneqq g_{p'} \in H$. If $h = 1$, we have $C \subseteq G'$ and hence $\bar{C} \subseteq D'$, so $|\bar{C}|$ is not divisible by $p$. Due to $C \not \subseteq Z(G')$, $p$ divides~$|C|$ (see Remark~\ref{rem:pdividesc}), which is a contradiction to $C \in \DCLab$ by \eqref{eq:dcl}. Hence $h \neq 1$. By \cite[Lemma~3.16]{BRE222}, applied with $N = G''$ and using $\nu(b_C) = \nu(C^+) \neq 0$ due to $C \in \DCLab$, we have $C \subseteq C_G(G'')$, which implies $C' \subseteq C_G(G'')$. Now \cite[Lemma~3.30]{BRE222}, applied to $N = G''$, yields $C' \in \DCLab$. Moreover, by \cite[Corollary~3.23]{BRE222}, we have $\Ann_{ZFD}(\bar{C}'^+) \subseteq \Ann_{ZFD}(\bar{C}^+)$.
Summarizing, this yields 
\begin{alignat*}{1}
\Ann_{ZFD}(\SIDCLab) &= \bigcap_{\substack{C \subseteq Z(G'), \\ C \not \subseteq G''}} \Ann_{ZFD}(b_{\bar{C}}) \cap \bigcap_{\substack{C \in \DCLab,\\ C \not \subseteq Z(G')}} \Ann_{ZFD}(b_{\bar{C}}) \\
&= \bigcap_{\substack{C \subseteq Z(G'), \\ C \not \subseteq G''}} \Ann_{ZFD}(b_{\bar{C}}) \cap \bigcap_{\substack{\text{$p'$-conjugacy classes} \\ C \in \DCLab}} \Ann_{ZFD}(b_{\bar{C}}) \\
&=\bigcap_{\substack{C \subseteq Z(G'), \\ C \not \subseteq G''}} \Ann_{ZFD}(\bar{C}^+ - |\bar{C}| \cdot 1) \cap \bigcap_{\substack{\text{$p'$-conjugacy classes} \\ C \in \DCLab}} \Ann_{ZFD}(\bar{C}^+).
\end{alignat*}
Now let $C \in \DCLab$ be a $p'$-conjugacy class. By the first part, $C$ contains an element $h \in H \setminus \{1\}$. 
By \cite[Lemma~3.22]{BRE222}, we have $\bar{C} = U_{D,h} \cdot \bar{h}$ and $U_{D,h}$ is a nontrivial normal subgroup of~$D$. Moreover, $\Ann_{ZFD}(\bar{C}^+) = \Ann_{ZFD}(U_{D,h}^+)$.
\end{proof}

%

\section{\texorpdfstring{Structure of $G/G''$}{Structure of G/G''}}\label{sec:nec}
Throughout this section, let $F$ be an algebraically closed field of characteristic $p > 0$ and let $G$ be a finite basic group for which $\soc(ZFG)$ is an ideal in $FG$. 
As before, we write $G = G' \rtimes H$ for a Hall $p'$-subgroup~$H$ of $G$. 
In this section, we study the structure of the quotient group $D \coloneqq G/G''$. Let $Z_D \coloneqq Z(G')/G''$. Our aim is to prove the following result:
\setcounter{theoremintro}{0}
\begin{theoremintro}
	The group $D$ has the following properties:
	\begin{enumerate}[(i)]
		\item $D' = T \times Z_D$ for an elementary abelian normal subgroup $T$ of $D$.
		\item Write $T = T_1 \times \dots \times T_n$ for minimal normal subgroups $T_1, \ldots, T_n$ of $D$. Then
		\[H/C_H(T) \cong H/C_H(T_1) \times \dots \times H/C_H(T_n).\]
		Moreover, for $i \in \{1,\ldots, n\}$, the group $H/C_H(T_i)$ is cyclic of order $|T_i|-1$.
	\end{enumerate} 
\end{theoremintro}

Parts (i) and (ii) of Theorem~\ref{theo:uniquesolutiond} will be proven in Sections~\ref{sec:decompdd} and~\ref{sec:actionh}, respectively. 

\subsection{\texorpdfstring{Decomposition of $D'$}{Decomposition of D'}}\label{sec:decompdd}

In this section, we derive the direct product decomposition of $D'$ given in part (i) of Theorem~\ref{theo:uniquesolutiond}. It forms the basis for all further results. Recall that by Lemma~\ref{lemma:nilpotencyclassofdg}, we have $G'' \subseteq Z(G')$. In the following, we 
identify~$H$ with its image in~$D$. 
\medskip

Let $J(G')$ denote the Thompson subgroup of $G'$, that is, the subgroup of $G'$ generated by all abelian subgroups of maximal order. The following result on $J(G')$ might be of independent interest: 

\begin{lemma}\label{lemma:j}
We have $J(G') = G' = MZ(G')$ for $M \coloneqq \{x \in G' \colon x^p \in G'' \} \trianglelefteq G$.
\end{lemma}

\begin{proof}
Set $J \coloneqq J(G')$ and $\hat{M}  \coloneqq \left\{x \in [J,G] \colon x^p \in [J,[J,G]]\right\}$. 
Then \cite[Lemma~3.18]{BRE222}, applied with $N = J$, yields 
\begin{equation}\label{eq:eqj}
G' = C_{G'}(J) \hat{M} = C_{G'}(J) J
\end{equation}
(using $\hat{M} \subseteq J \subseteq G'$).
Observe that $C_{G'}(J) \subseteq J$ as every element in $C_{G'}(J)$ centralizes every abelian subgroup $A$ of maximal order in $G'$ and is thus contained in $A$ by maximality. This implies $G' = J$. Using $G' =[G'H,G'H] = [G',G'H]= [G',G]$, this yields $\hat{M} = M$, and then $G' = Z(G') M = M Z(G')$ follows by~\eqref{eq:eqj}.
\end{proof}


With this, we obtain the desired decomposition of $D'$:

\begin{theorem}\label{theo:structuredgmoddp}
There exists an elementary abelian normal subgroup $T$ of $D$ with $D' = T \times Z_D$.  
\end{theorem}

\begin{proof}
By Lemma~\ref{lemma:j}, we have $G' = MZ(G')$ with $M \coloneqq \{x \in G' \colon x^p \in G'' \} \trianglelefteq G$. 
Note that $M/G''$ is elementary abelian. By Maschke's theorem, there exists a normal subgroup $L$ of $G$ with $G'' \subseteq L $ such that $M/G'' = L/G'' \times C_{M}(G')/G''$ holds. With this, we obtain $$D' =G'/G'' = (M/G'') (Z(G')/G'') = L/G'' \times Z(G')/G''.$$ 
Setting $T \coloneqq L/G''$ proves the statement. 
\end{proof}

This proves the first part of Theorem~\ref{theo:uniquesolutiond}. 

\begin{Remark}\label{rem:di2notpossible}
	Observe that the subgroups of $D'$ which are normal in $D$ are precisely the $\F_p H$-submodules of $D'$. In particular, the minimal normal subgroups of $D$ contained in $T$ are the simple $\F_p H$-submodules of~$T$. 
In the situation of Theorem~\ref{theo:structuredgmoddp}, $T$ is a semisimple $\F_p H$-module. Hence there exist $n \in \N_0$ and simple $\F_p H$-modules $T_1, \ldots, T_n$ with 
\begin{equation*}
T = T_1 \times \dots \times T_n.
\end{equation*} 
Due to $C_{D'}(H) = 1$ (see Remark~\ref{rem:trivia}), $T_i$ is a nontrivial $\F_p H$-module for all $i \in \{1, \ldots,n\}$. In particular, we have $|T_i| \geq 3$. 
\end{Remark}

Until the end of this section, let $T$ be as in Theorem~\ref{theo:structuredgmoddp} and fix a decomposition $T = T_1 \times \dots \times T_n$ as in Remark~\ref{rem:di2notpossible}.


\begin{Example}\label{ex:sl}
	Let $F$ be a field of characteristic 2 and set $G \coloneqq \SL_2(\F_3)$. We have $G = G' \rtimes H$ with $G' \cong Q_8$ and $H \cong C_3$. Moreover, $Z(G) \eqqcolon \langle z \rangle$ is cyclic of order 2. 
	\medskip
	
	We first show that $\soc(ZFG)$ is an ideal in $FG$. For $h \in H$, we set $S_h \coloneqq \soc(ZFG) \cap FhG'$. Since $\soc(ZFG) = \bigoplus_{h \in H} S_h$ holds by \cite[Remark~3.15]{BRE222} and we have $(hG')^+ \in S_h$ for all $h \in H$, it suffices to show $\dim_F S_h = 1$. The derived subgroup $G'$ decomposes into the conjugacy classes $\{1\}$, $Z(G) \setminus \{1\}$ and $G' \setminus Z(G)$, which easily yields $S_1 \subseteq FZ(G)^+ + F(G' \setminus Z(G))^+$. For $h \in H \setminus \{1\}$, the coset $hG'$ consists of the conjugacy classes $[h]_G$ and $[hz]_G$, so $S_h \subseteq F[h]_G^+ + F[hz]_G^+$ follows. Since $Z(G)^+ \cdot [h]_G^+ = (hG')^+ \neq 0$ holds for all $h \in H \setminus \{1\}$, we obtain $\dim_F S_h = 1$ for all $h \in H$.  
	\medskip
	
	The group $H$ acts transitively on the nontrivial elements of $D' \cong C_2 \times C_2$. In particular, $D'$ is a simple $\F_2 H$-module, so we obtain $n = 1$, $T = D'$ and $Z_D = 1$ in the decompositions given in Theorem \ref{theo:structuredgmoddp} and Remark~\ref{rem:di2notpossible}.
\end{Example}

\subsection{\texorpdfstring{Action of $H$}{Action of H}}\label{sec:actionh}
In this section, we study the action of $H$ on the group $D'$, and in particular, on its subgroup $T$ defined in the previous section. Our aim is to prove part (ii) of Theorem~\ref{theo:uniquesolutiond}. 
Throughout, we use the notation from Theorem~\ref{theo:structuredgmoddp} and Remark~\ref{rem:di2notpossible}. Again, we set $\bar{g} \coloneqq g G'' \in D$ for every $g \in G$ (similarly for subsets of $G$). 
\medskip

The starting point is the following technical lemma:

\begin{lemma}\label{lemma:ficyclic}
For $i \in \{1, \ldots, n\}$, let
$N_i \coloneqq \prod_{j \neq i} T_j \times Z_D \trianglelefteq D$ and consider the preimage $M_i$ of $N_i$ in~$G$. Then the following hold:
\begin{enumerate}[(i)]
\item The centralizer $C_H(M_i)$ is nontrivial. 
\item For every $h \in C_H(M_i) \setminus \{1\}$, we have $[h]_G \in \DCLab$ and $[\bar{h}]_D = T_i \cdot \bar{h}$.
\end{enumerate} 
\end{lemma}

\begin{proof}
$\null$
\begin{enumerate}[(i)]
\item Note that $D' = T_i \times N_i $. As $T_i \neq 1$, we have $N_i^+ \notin (D')^+ \cdot FD$, which implies $N_i^+ \notin \Ann_{ZFD}(\SIDCLab)$ by Theorem~\ref{theo:anndecconjconstantcoeffs}. Hence there exists $C \in \DCLab$ with $N_i^+ \cdot b_{\bar{C}} \neq 0$, where $b_{\bar{C}}$ denotes the basis element of $J(ZFD)$ corresponding to $\bar{C}$ (see Lemma~\ref{lemma:structjzfg}).
If $C \subseteq Z(G')$, then $\bar{C} \subseteq Z_D$. This leads to the contradiction $N_i^+ \cdot b_{\bar{C}} =N_i^+ \cdot (\bar{C}^+ - |\bar{C}| \cdot 1) = |\bar{C}| \cdot N_i^+ - |\bar{C}| \cdot N_i^+ = 0$, as $N_i$ contains~$Z_D$. By Lemma~\ref{lemma:pstrichkonjugation}, we can now assume that $C$ is a $p'$-conjugacy class. Hence we have $C = [h]_G$ for some $h \in H \setminus \{1\}$. Moreover, \cite[Lemma~3.30]{BRE222} yields $h \in C_G(G'')$ and by \cite[Lemma~3.16]{BRE222}, we have $\bar{h} \in C_H(N_i)$. By \cite[Theorem~5.3.2]{GOR68}, we then obtain $h \in C_H(M_i)$, so this group is nontrivial.

\item Let $h \in C_H(M_i) \setminus \{1\}$ and set $C \coloneqq [h]_G$. Then $\bar{C} = U_{D,h} \cdot \bar{h}$ (see Lemma~\ref{lemma:pstrichkonjugation}). Due to $D' = T_i \times N_i$ and since $\bar{h}$ centralizes $N_i$, we obtain $U_{D,h} \subseteq T_i$ (see Remark~\ref{rem:trivia}) as $H$ normalizes $T_i$. On the other hand, since $D'$ is abelian, the map $\alpha \colon T_i \to T_i,\ t \mapsto [t,\bar{h}]$ is a homomorphism of $\F_p H$-modules. Due to $\bar{h} \notin C_H(D')$ (see Remark~\ref{rem:centralizerhdd}), $\alpha$ is an isomorphism by Schur's lemma, which yields $T_i = U_{D,h}$. 
Furthermore, $h \in C_H(M_i) \subseteq C_H(G'')$ implies $C \in \DCLab$ by \cite[Lemma~3.30]{BRE222}. \qedhere
\end{enumerate}
\end{proof}

\begin{Example}
For the group $G = \SL_2(\F_3)$ from Example~\ref{ex:sl}, we have $M_1 = G'' = Z(G)$ and hence $C_H(M_1) = H$ for any Hall $2'$-subgroup $H$ of $G$. 
\end{Example}

The proof of Theorem~\ref{theo:uniquesolutiond}\,(ii) will rely on Theorem~\ref{theo:anndecconjconstantcoeffs}. In order to simplify the calculation of the occurring annihilators, we first improve on Lemma~\ref{lemma:pstrichkonjugation}:

\begin{lemma}\label{lemma:annd}
We have \[\Ann_{ZFD}(\SIDCLab) = \bigcap_{\substack{\bar{C} \in \Cl(D) \setminus \{1\} \\ \bar{C} \subseteq Z_D}} \Ann_{ZFD}(b_{\bar{C}}) \cap \bigcap_{i = 1}^n \Ann_{ZFD}(T_i^+).\] 

\end{lemma}

\begin{proof}
Set \[M \coloneqq \left\{b_{\bar{C}} \colon \bar{C} \in \Cl(D) \setminus \{1\},\, \bar{C} \subseteq Z_D\right\} \cup \{T_1^+, \ldots, T_n^+\},\] 
so we need to show $\Ann_{ZFD}(\SIDCLab) = \Ann_{ZFD}(M)$.
\medskip

First let $y \in \Ann_{ZFD}(M)$. Since $\bar{C} \subseteq Z_D$ for any $C \in \Cl(G)$ with $C \subseteq Z(G')$, we need to show $y \cdot \bar{C}^+ = 0$ for every $p'$-conjugacy class $C \in \DCLab$ (see Lemma~\ref{lemma:pstrichkonjugation}). By Lemma~\ref{lemma:pstrichkonjugation}, we have $\bar{C} = U_{D,h} \cdot \bar{h}$ for some $h \in H$ and $U_{D,h}$ is a nontrivial normal subgroup of $D$ contained in $D'$. It suffices to show $y \cdot U_{D,h}^+ = 0$. If $[h, T_i] \neq 1 $ for some $i \in \{1, \ldots, n\}$, then $T_i \subseteq U_{D,h}$ as $T_i$ is a minimal normal subgroup of $D$. In particular, $U_{D,h}$ is a union of cosets of $T_i$ and hence $y \cdot U_{D,h}^+ = 0$ follows. Assume otherwise that $[h,T_i] = 1$ for every $i \in \{1, \dots, n\}$, that is, $h \in C_H(T)$. Since $h$ acts nontrivially on~$D'$ by Remark~\ref{rem:trivia}\,(iii), this implies $[h,Z_D] \neq 1$ and hence $U_{D,h} \cap Z_D \neq 1$. Then \[(U_{D,h} \cap Z_D)^+ = (U_{D,h} \cap Z_D)^+ - |U_{D,h} \cap Z_D| \cdot 1 = \sum_{\substack{K \in \Cl(D), \\ K \subseteq U_{D,h} \cap Z_D}} K^+ - |K| \cdot 1 \in FM\] is annihilated by $y$ and hence $y \cdot U_{D,h}^+ = 0$ follows as $U_{D,h}$ is a union of cosets of $U_{D,h} \cap Z_D$.
\medskip

Conversely, let $y \in \Ann_{ZFD}(\SIDCLab)$. For $C \in \Cl(G)$ with $C \subseteq Z(G')$ and $C \not \subseteq G''$, we have $C \in \DCLab$ by \eqref{eq:dcl}. In particular, $y$ annihilates $b_{\bar{C}}$ for all $\bar{C} \in \Cl(D) \setminus \{1\}$ with $\bar{C} \subseteq Z_D$. Now let $i \in \{1, \ldots, n\}$. By Lemma~\ref{lemma:ficyclic}, there exists an element $h \in C_H(M_i) \setminus \{1\}$, for which we have $[h]_G\in \DCLab$ as well as $[\bar{h}]_D = T_i \cdot \bar{h}$. The condition $y \cdot [\bar{h}]_D^+ = 0$ then translates to $y \cdot T_i^+ = 0$, which proves $y \in \Ann_{ZFD}(M)$. 
\end{proof}



The following observation plays a crucial role in our arguments: 

\begin{lemma}\label{lemma:claim}
For $i \in \{1, \ldots, n\}$, let $t_i, t_i' \in T_i$ such that either $t_i = t_i' = 1$ or $t_i \neq 1 \neq t_i'$. Then $t \coloneqq t_1 \cdots t_n$ and $t' \coloneqq t_1' \cdots t_n'$ are conjugate in $D$.
\end{lemma}

\begin{proof} We proceed by induction on $m \coloneqq |\{i \in \{1, \ldots, n\} \colon t_i \neq 1\}|.$ We may assume $t_i \neq 1$ for $i \in \{1, \ldots, m\}$ and $t_{m+1} = \ldots = t_n = 1$. For $m = 0$, there is nothing to show. For $m \in \{1, \ldots, n\}$, we consider the subgroup $W \coloneqq T_{m+1} \times \dots \times T_n \times Z_D \trianglelefteq D$. In the following, let $[t]  \coloneqq [t]_D$ and $[t'] \coloneqq [t']_D$. Set $\ell \coloneqq |[t]|$, $\ell' \coloneqq |[t']|$ and $y \coloneqq \left(\ell' [t]^+ - \ell [t']^+\right) \cdot W^+$. Both $\ell$ and $\ell'$ are not divisible by $p$ as $D'$ is abelian (see Remark~\ref{rem:pdividesc}). Using Lemma~\ref{lemma:annd}, we now show $y \in \Ann_{ZFD}(\SIDCLab).$ Note that $(T_1 \times \dots \times T_m) \cap W= 1$. By construction, we have $y \in ZFD.$ Let $\nu_W \colon FD \to F[D/W]$ denote the canonical projection. If $\bar{C} \in \Cl(D) \setminus \{1\}$ is a conjugacy class with $\bar{C} \subseteq Z_D$, then $\bar{C} \subseteq W$ and hence $W^+ \bar{C}^+ = |C| W^+$. Thus $W^+ \cdot b_{\bar{C}} = 0$ and hence $y \cdot b_{\bar{C}} = 0$. 
	By Lemma~\ref{lemma:annd}, it remains to show $y \cdot T_i^+ = 0$ for every $i \in \{1, \ldots, n\}$.
	\medskip
	
	First let $i \in \{1, \ldots, m\}$. By induction, we obtain $C_i \coloneqq [t_1 \cdots t_{i-1} \cdot t_{i+1} \cdots t_m]_D = [t_1' \cdots t_{i-1}' \cdot t_{i+1}' \cdots  t_m']_D$. We have 
	$[t]^+ \cdot T_i^+ = \frac{\ell}{|C_i|} \cdot C_i^+ \cdot T_i^+$ and a similar formula holds for $[t']^+ \cdot T_i^+$. This yields
	$$\left(\ell' [t]^+ - \ell [t']^+\right) \cdot T_i^+ = \ell'\cdot \frac{\ell}{|C_i|} \cdot C_i^+ \cdot T_i^+- \ell \cdot \frac{\ell'}{|C_i|} \cdot C_i^+ \cdot T_i^+ = 0$$ and hence $y \cdot T_i^+ = 0$ follows. Now let $i \in \{m+1, \ldots, n\}$. Since $T_i \subseteq W$, we obtain $\nu_{W}(T_i^+) = 0$ and hence $y \cdot T_i^+ = 0$. This shows $y \in \Ann_{ZFD}(\SIDCLab)$.
If $[t] \neq [t']$, we have $y \notin (D')^+ \cdot FD$, which is a contradiction to Theorem~\ref{theo:anndecconjconstantcoeffs}. Hence $t$ and $t'$ are conjugate in~$D$. 
\end{proof}

\begin{Remark}\label{rem:transitive}
By Lemma~\ref{lemma:claim} and using that $D'$ is abelian, $H$ acts transitively on $T_i \setminus \{1\}$ for $i \in \{1, \dots, n\}$.
\end{Remark}

Now we can prove the second part of Theorem~\ref{theo:uniquesolutiond}:
\begin{theorem}\label{theo:rho}
The canonical map 
\begin{alignat*}{1}
\rho \colon H/C_H(T) &\to  H/C_H(T_1) \times \dots \times H/C_H(T_n) \\
h C_H(T) &\mapsto (h C_H(T_1), \ldots, h C_H(T_n))
\end{alignat*}
 is an isomorphism. 
\end{theorem}

\begin{proof} 
It suffices to show the surjectivity of the canonical projection
\begin{alignat*}{1}
	\rho' \colon H &\to  H/C_H(T_1) \times \dots \times H/C_H(T_n) \\
	h &\mapsto (h C_H(T_1), \ldots, h C_H(T_n))
\end{alignat*}
since the statement then follows from the isomorphism theorem. 
Since $T_i$ is a simple $\F_p H$-module for every $i \in \{1, \ldots, n\}$ and $H$ is abelian, $H/C_H(T_i)$ is cyclic (see \cite[Theorem~3.2.3]{GOR68}). Let $x_i \in H$ such that $ x_i C_H(T_i)$ generates $H / C_H(T_i)$ 
and fix $t_i \in T_i \setminus \{1\}$. By Lemma~\ref{lemma:claim}, $t_1 \cdots t_n$ and $x_1 t_1 x_1^{-1} \cdot t_2 \cdots t_n$ are conjugate in~$D$. Since $D'$ is abelian, there exists an element $h \in H$ with
	$h (t_1 \cdots t_n) h^{-1} = x_1 t_1 x_1^{-1} \cdot t_2 \cdots t_n.$ This implies $h t_1 h^{-1} = x_1 t_1x_1^{-1}$ as well as $h t_i h^{-1} = t_i$ for all $i \in \{2, \ldots, n\}$. For every $j \in \{1, \ldots, n\}$, the nontrivial elements in $T_j$ are conjugate by elements in $H$ by Remark~\ref{rem:transitive}, so $C_H(t_j) = C_H(T_j)$ follows as $H$ is abelian. With this, we obtain
	\[\rho'(h) = (x_1 C_H(T_1), C_H(T_2), \ldots, C_H(T_n)).\] Using a similar argument for all indices, we see that $\rho'$ is surjective and hence $\rho$ is an isomorphism. 
\end{proof}

This completes the proof of Theorem~\ref{theo:uniquesolutiond}. It yields the following convenient representation of $H$: 

\begin{Remark}\label{rem:ei}
Let $i \in \{1, \ldots,n \}$. By Theorem~\ref{theo:rho}, there exists an element $e_i \in H$ centralizing $T_j$ for $j \neq i$ such that $e_i C_H(T_i)$ generates $H/C_H(T_i)$. Moreover, $\langle e_i \rangle$ acts transitively on $T_i \setminus \{1\}$ (see Remark~\ref{rem:transitive}). Observe that $H = \langle e_1, \ldots, e_n, C_H(T)\rangle$. 
\end{Remark}

We conclude this section by applying Theorem~\ref{theo:uniquesolutiond} to the case of Frobenius groups. 

\begin{Example}[Frobenius groups]
	Let $G$ be an arbitrary finite Frobenius group with Frobenius kernel $K$ and Frobenius complement $A$. First suppose that $\soc(ZFG)$ is an ideal of $FG$. Then $G = P \rtimes H$ for $P \in \Syl_p(G)$ and an abelian $p'$-group $H$ (see \cite[Theorem~1]{BRE222}). It is easy to see that $P = K$ and that we may choose $H = A$. In particular, $H$ is cyclic. We have $P = C_P(H) G' = G'$ by \cite[Theorem 5.3.5]{GOR68} and since~$G$ is a Frobenius group. Moreover, $O_{p'}(G)= 1$, so $G$ is basic. Again, we consider the group $D = G/G''$ and write $D' = T_1 \times \dots \times T_n \times Z_D$ with $Z_D = Z(G')/G''$ as in Theorem \ref{theo:uniquesolutiond}. If $n \geq 1$, we have $C_H(G'') \neq 1$ by Lemma~\ref{lemma:ficyclic} (as the subgroup $M_1$ defined therein contains $G''$). This is a contradiction. Hence $n = 0$, so~$G'$ is abelian. Conversely, if $G$ is a Frobenius group with an abelian Frobenius kernel $G' \in \Syl_p(G)$, then $\soc(ZFG) \trianglelefteq FG$ by \cite[Theorem 3]{BRE222}. 
\end{Example}

\section{\texorpdfstring{Classification}{Classification}}\label{sec:generalgroupsfirstcase}
Let $F$ be an algebraically closed field of characteristic $p > 0$ and let $G$ be a finite basic group.
Recall that $G'' \subseteq Z(G')$. In this chapter, we study the case $Z(G')= G''$. This is motivated by the decomposition given in Theorem~\ref{theo:uniquesolutiond}, in which this condition translates to $Z_D = 1$. Our aim is to characterize the groups~$G$ of this form with $\soc(ZFG) \trianglelefteq FG$ (see Theorem~\ref{theo:b}). 
\medskip

In Section~\ref{sec:decomposition}, we show that every finite basic group $G$ with $Z(G) = G'$ for which $\soc(ZFG)$ is an ideal in~$FG$ has a central product decomposition as given in Theorem~\ref{theo:b}. 
In Section~\ref{sec:proofc}, we consider the converse implication.
Combining these results, we prove Theorem~\ref{theo:b} in Section~\ref{sec:prooftheob}.

\subsection{Decomposition}\label{sec:decomposition}
Throughout this subsection, let $G$ be a finite basic group with $Z(G') = G''$ for which $\soc(ZFG)$ is an ideal in $FG$. We show that $G$ admits a central product decomposition as given in Theorem~\ref{theo:b}. To this end, we proceed in two steps. In Section~\ref{sec:qi}, we decompose~$G$ into a central product of basic groups $G_1, \ldots, G_n$ satisfying $Z(G_i) = G_i''$ and $\soc(ZFG_i) \trianglelefteq FG_i$ (for $i = 1, \ldots, n$). 
Additionally, $G_i/G_i''$ is a 1-dimensional affine linear group. 
In Section~\ref{sec:specialcasetsimple}, we study these groups individually and prove that each of them satisfies the conditions given in Theorem~\ref{theo:b}\,(ii).

\subsubsection{Central products}\label{sec:qi}
Recall that $\AGL(1,p^d)$ denotes the 1-dimensional affine linear group over the field with $p^d$ elements. In this section, we prove the following result:

\begin{theorem}\label{theo:qi}
Let $G$ be a finite basic group with $Z(G') = G''$ for which $\soc(ZFG)$ is an ideal in $FG$. 
There exist basic subgroups $G_1, \ldots, G_n$ of $G$ with $G = G_1 * \dots * G_n$ such that, for all $i \in \{1, \ldots, n\}$, the following hold: 
	\begin{enumerate}[(i)]
		\item $\soc(ZFG_i)$ is an ideal in $FG_i$. 
		\item $Z(G_i') = G_i''$.
		\item $G_i/G_i'' \cong \AGL(1,|G_i'/G_i''|)$.
	\end{enumerate} 
\end{theorem}

Note that by Lemma~\ref{lemma:nilpotencyclassofdg}, $G'$ is a special $p$-group (in the sense of \cite[page 183]{GOR68}). The structure of $D \coloneqq G/G''$ can be described as follows:

\begin{Remark}\label{rem:hdirectproduct}\label{rem:eitrivial}
By Theorem \ref{theo:uniquesolutiond}, we have $D' = T_1 \times \dots \times T_n$ for some $n \in \N_0$ and minimal normal subgroups $T_1, \ldots, T_n$ of $D$. By Remark~\ref{rem:centralizerhdd}, we have $C_H(D') = 1$. Hence $ H = \langle e_1 \rangle \times \dots \times \langle e_n \rangle$, where $e_i \in H$ centralizes~$T_j$ for $j \neq i$ and acts regularly on $T_i \setminus \{1\}$ for every $i \in \{1, \ldots, n\}$ (see Theorem~\ref{theo:uniquesolutiond} and Remark~\ref{rem:ei}). In particular, $s_i \coloneqq \ord(e_i) = |T_i|-1$. 
\end{Remark}

This yields the following decomposition of $D$:

\begin{lemma}\label{lemma:structuredagl}
We have $D \cong \AGL(1,|T_1|) \times \dots \times \AGL(1,|T_n|)$.
\end{lemma}

\begin{proof}
	For every $i \in \{1, \ldots, n\}$, we set $A_i \coloneqq \langle e_i, T_i \rangle$. Then $D = \langle A_1, \dots, A_n \rangle$. For $i,j \in \{1, \ldots, n\}$ with $i \neq j$, we have $[A_i, A_j] = 1$ (see Remark~\ref{rem:hdirectproduct}) and $A_i \cap A_j \subseteq C_{D'}(H) = 1$ (see Remark~\ref{rem:centralizerhdd}). Moreover, every element of $D$ can be expressed as $a_1 \cdots a_n$ with $a_i \in A_i$ for $i =1, \ldots, n$ in a unique way. This implies $D = A_1 \times \dots \times A_n$. Since $\langle e_i \rangle$ acts regularly on $T_i \setminus \{1\}$, we obtain $A_i \cong \AGL(1, |T_i|)$ for all $i \in \{1, \ldots, n\}$. 
\end{proof}

As before, we set $\bar{g} \coloneqq g G'' \in D$ for all $g \in G$. For $i \in \{1, \ldots, n\}$, let $M_i$ denote the preimage of $N_i = \prod_{j \neq i} T_i$ in~$G$. For $h \in H$, recall that $[h]_G = U_{G,h}h$ with $U_{G,h} = \{[g,h] \colon g \in G\}$. Moreover, recall that $C_H(M_i)$ is nontrivial by Lemma~\ref{lemma:ficyclic}.
For the desired decomposition of $G$, we need the following technical lemma on the structure of the conjugacy classes of elements in $C_H(M_i)$:

\begin{lemma}\label{lemma:realconj}\label{lemma:fiallconjugate}
	Let $i \in \{1, \ldots, n\}$ and fix an element $f \in G'$ with $\bar{f} \in T_i \setminus \{1\}$. Let $h \in C_H(M_i) \setminus \{1\}$ and set $g_i \coloneqq g_{h,i} = [f,h]$. 
	\begin{enumerate}[(i)]
		\item The element $g_i$ has the following properties:
		\begin{enumerate}[(a)]
			\item $U_{G,h} = \{1\} \cup \left\{e_i^k g_{i} e_i^{-k} \colon 0 \leq k \leq s_i-1\right\}$,
		\item $[g_i, e_j] = [g_i, g_j]= 1$ for all $j \in \{1, \ldots, n\}$ with $j\neq i$,
		\item $g_i$ is $G$-conjugate to $g_i^{-1}$ as well as to $[e_i^k, g_i]$ for every $k \in \{1, \ldots, s_i-1\}$.
	\end{enumerate}
		\item For $h, h' \in C_H(M_i) \setminus \{1\},$ the elements $g_{h,i}$ and $g_{h',i}$ are conjugate in $G.$ 
	\end{enumerate}
\end{lemma}

\begin{proof} 
We write $x\sim y$ if $x,y \in G$ are conjugate in $G$. Since $\langle e_i \rangle$ acts transitively on $T_i \setminus \{1\}$, we have 
	\begin{equation}\label{eq:labelcosetsofn}
T_i = \{1\} \cup \left\{e_i^k \bar{f} e_i^{-k} \colon 0 \leq k \leq s_i-1\right\}.
	\end{equation}

	\begin{enumerate}[(i)]	
	\item Set $C \coloneqq [h]_G$. 
			  \begin{enumerate}[(a)]
		\item As $H \subseteq C_G(h)$ holds, $|C|$ is a power of~$p$. By Lemma \ref{lemma:ficyclic} and \eqref{eq:dcl}, this yields $|C| = |\bar{C}| = |T_i|$. For $R \coloneqq \{1\} \cup \left\{e_i^k f e_i^{-k} \colon 0 \leq k \leq s_i-1\right\}$, we have $\bar{R} = T_i$ by \eqref{eq:labelcosetsofn}. Since $h$ acts on $T_i \setminus \{1\}$ without fixed points and centralizes $\overline{M_i} = \prod_{j \neq i} T_j$ as well as $H$, the elements in $\bar{R}$ form a set of representatives for the cosets of $C_D(h)$ in $D$. As $|C| = |\bar{C}|$ holds, $R$ is a set of representatives for the cosets of $C_G(h)$ in~$G$. Using that $e_i$ and $h$ commute, we obtain
		\begin{equation*}
		\begin{aligned}
		U_{G,h} = \{[g,h] \colon g \in R\} = \left\{[e_i^k f e_i^{-k},h] \colon 0 \leq k \leq s_i-1\right\} \cup \{1\} = \left\{e_i^k [f,h] e_i^{-k} \colon 0 \leq k \leq s_i-1\right\}  \cup \{1\}.
		\end{aligned}
		\end{equation*}
		
As $\bar{h}$ acts on $T_i \setminus \{1\}$ without fixed points, we have $[\bar{f},\bar{h}] \neq 1$, which yields $g_i \notin G''$.
		\item Let $j \in \{1, \ldots,n \}$ with $j \neq i$. Due to $e_j \in C_H(T_i)$, we have $e_j g_i e_j^{-1} \in g_i G''$, which yields $e_j g_i h e_j^{-1} = e_j g_i e_j^{-1} h \in [h]_G \cap g_i h G'' = \{g_i h \}.$ Here, we use that every coset of $G''$ contains at most one element of $C$ due to $|C| = |\bar{C}|$ (see (a)). This yields $e_j g_i e_j^{-1} = g_i.$ Similarly, one shows that $[g_i, g_j] = 1$ as $g_j \in M_i$ commutes with $h$.

		\item
		Let $u \in U_{G,h}$ and write $u = [a,h]$ with $a \in G$. Then $U_{G, h} h = C = a C a^{-1} = a U_{G, h} a^{-1} \cdot a ha^{-1}.$
		Since the elements in $U_{G, h} \setminus \{1\}$ are conjugate in $G$ by (a), this also holds for the non\-tri\-vial elements in $a U_{G, h} a^{-1} = U_{G, h} [a,h]^{-1} = U_{G, h} u^{-1}.$ Now choose two distinct nontrivial elements $u, u' \in U_{G, h}$, which is possible due to $|U_{G, h}| = |T_i| \geq 3$ (see Remark~\ref{rem:di2notpossible}). By (a), $u$ and $u'$ are conjugate in~$G$. Using that the nontrivial elements in $U_{G, h} u^{-1}$ and $U_{G, h} u'^{-1}$ are conjugate in~$G$, we obtain $u' u^{-1} \sim u^{-1} \sim u'^{-1} \sim u u'^{-1} = (u' u^{-1})^{-1}$. Hence $[u^{-1}]_G$, so also $[u]_G$, is a real conjugacy class. Setting $u = g_i$ and $u' = e_i^k g_i e_i^{-k}$ for some $k \in \{1, \ldots, s_i-1\}$, we obtain $g_i^{-1} \sim g_i \sim [e_i^k,g_i]$.
			\end{enumerate}
		\item Set $g_i \coloneqq g_{h,i}$. There exists $k \in \Z$ such that $g_i' \coloneqq e_i^k g_{h',i} e_i^{-k}$ is contained in $g_i G''$. Set $c \coloneqq g_i^{-1} g_i' \in G''$. Now write $h C_H(T_i) = e_i^k C_H(T_i)$ for some $k \in \Z$. As both $h$ and $e_i^k$ centralize $T_j$ for $j \neq i$, we have $h e_i^{-k} \in C_H(D') = 1$ (see Remark~\ref{rem:trivia}), so $h = e_i^k$. Since $h$ centralizes $G''$, using part (c) of (i) yields $g_{h',i} \sim g_i'\sim [h, g_i'] = [h, g_i c] = [h, g_i] \sim g_i$. \qedhere
	\end{enumerate}
\end{proof}

With this, we can now prove Theorem~\ref{theo:qi}:

%
%
%
%

\begin{proof}[Proof of Theorem~\ref{theo:qi}]
For every $i \in \{1, \ldots,n\}$, we fix an element $h_i \in C_H(M_i) \setminus \{1\}$ (this is possible by Lemma~\ref{lemma:ficyclic}) and let $g_i \coloneqq g_{h_i,i}$ be defined as in Lemma~\ref{lemma:realconj}. We set $G_i \coloneqq \langle g_i, e_i \rangle$ for all $i \in \{1, \ldots,n\}$. For $i \neq j$, we have $[G_i, G_j] = 1$ by Lemma~\ref{lemma:realconj}. Set $Q \coloneqq \langle G_1, \ldots, G_n \rangle$. Using Theorem~\ref{theo:uniquesolutiond}, it is easily seen that every element in $G'$ can be expressed as $u_1 \cdots u_n  g$ with $u_i \in \{1\} \cup \{e_i^k g_i e_i^{-k} \colon k \in \Z\}$ for $i \in \{1, \ldots, n\}$ and some $g \in G''$. Due to $G'' \subseteq \Phi(G')$, the derived subgroup~$G'$ is generated by the elements $e_i^k g_i e_i^{-k}$ for $i \in \{1, \ldots, n\}$ and $k \in \Z$. Hence $G' \subseteq Q$ follows. This yields $G = G' H \subseteq Q$ and hence $G = G_1 * \dots * G_n$.
\medskip

Let $i \in \{1, \ldots, n\}$. The elements $[e_i, g_i]$ and $g_i$ are conjugate in $G = G_1 * \dots * G_n$ by Lemma~\ref{lemma:realconj}, so also in~$G_i$. As $[e_i, g_i] \in G_i'$ holds, we have $g_i \in G_i'$. In particular, $e_i G_i'$ generates $G_i/G_i'$, so this is a $p'$-group. Hence $G_i'$ is a Sylow $p$-subgroup of $G_i$. By Lemma \ref{lemma:centralproduct}, we have $\soc(ZFG_i) \trianglelefteq FG_i$, which implies $G_i'' \subseteq Z(G_i')$ (see Lemma~\ref{lemma:nilpotencyclassofdg}). We now show the inclusion $Z(G_i') \subseteq G_i''$. Analogously to Remark~\ref{rem:centralizerhdd}\,(ii), one shows $C_{G_i'/G_i''}(e_i) = 1$, which implies $C_{G_i'}(e_i) \subseteq G_i''$. Since $G' = G_1' * \dots * G_n'$ holds, $Z(G_i') \subseteq Z(G') \cap G_i = G'' \cap G_i.$ Any $g \in G''$ can be expressed as $g = q a$ for some $q \in G_i''$ and $a \in \prod_{j \neq i} G_j'' \subseteq C_{G''}(e_i)$. If additionally $g \in G_i$, then $a = q^{-1} g \in C_{G_i'}(e_i)$.
Then $$Z(G_i') \subseteq G'' \cap G_i \subseteq C_{G_i'}(e_i) G_i'' = G_i''.$$
This shows $Z(G_i') =G_i''$. Moreover, $G_i'' = G_i \cap G''$. Setting $\bar{G}_i \coloneqq G_iG''/G''\subseteq D$ as before, this yields $$G_i/G_i'' = G_i/G_i \cap G'' \cong G_i G''/G'' = \bar{G}_i.$$ 
Set $A_i \coloneqq \langle e_i, T_i \rangle$ as in the proof of Lemma~\ref{lemma:structuredagl}. Clearly $\bar{G}_i \subseteq A_i$. On the other hand, $T_i$ is a minimal normal subgroup of $D$ and hence $T_i \subseteq \bar{G}_i$, which shows $\bar{G}_i = A_i$ and $\bar{G}_i' = T_i$. By Lemma~\ref{lemma:structuredagl}, we have $A_i \cong \AGL(1,|T_i|)$. 
\end{proof}


In the next subsection, we prove that each of the groups $G_1, \ldots, G_n$ satisfies the properties (a)-(c) stated in Theorem~\ref{theo:b}\,(ii) with $K_i = G_i'$. The main obstacle is to show that $G_i'$ is a Camina group for every $i \in \{1, \ldots, n\}$. 

\subsubsection{Camina groups}\label{sec:specialcasetsimple}
Throughout this subsection, let $G$ be a finite basic group satisfying $Z(G')= G''$ for which $\soc(ZFG)$ is an ideal in $FG$. Additionally, we assume $G/G'' \cong \AGL(1,|G'/G''|)$. As before, we write $G = G' \rtimes H$ and set $D \coloneqq G/G''$. The aim of this section is to show that $G$ has the properties (a)-(c) described in Theorem~\ref{theo:b}\,(ii) (with respect to the decomposition $G = G' \rtimes H$). To prove Theorem~\ref{theo:b}, these results will be applied to the groups $G_1, \ldots, G_n$ constructed in Theorem~\ref{theo:qi}.  We first observe the following:

\begin{Remark}\label{rem:theobabd}
Since $G$ is basic and $Z(G') = G''$, the group $G'$ has nilpotency class 2. By Remark~\ref{rem:hdirectproduct}, the group $H = \langle e_1\rangle$ is cyclic of order $s \coloneqq |D'|-1$ and permutes the nontrivial cosets of $G''$ in $G'$ transitively. By Lemma~\ref{lemma:ficyclic}, we have $C_H(G'') \neq 1$, so the action of $H$ on $G''$ is not faithful. 
\end{Remark}

Thus in fact, it remains to show the following statement:


\begin{theorem}\label{theo:CaminaDsinglenormal}
Let $G$ be a finite basic group satisfying $Z(G') = G''$ for which $\soc(ZFG)$ is an ideal in~$FG$. Additionally, we assume $G/G'' \cong \AGL(1,|G'/G''|)$. 
Then $G'$ is a Camina group.
\end{theorem}

The proof of this theorem occupies the remainder of this section. 

\begin{Example}\label{ex:slcont}
	Suppose that $F$ has characteristic 2 and let $G \coloneqq \SL_2(\F_3)$. Recall that $\soc(ZFG) \trianglelefteq FG$ (see Example~\ref{ex:sl}). The derived subgroup $G' \cong Q_8$ is a Camina group. In particular, the statement of Theorem~\ref{theo:CaminaDsinglenormal} holds in this case. 
\end{Example}

Recall that we have $C_{G}(G') = Z(G') = G''$ and hence $Z(G) \subseteq G''$ follows. In order to simplify our arguments, we use the following observation on $Z(G)$:

\begin{lemma}\label{lemma:center}
Unless $p = 2$ and $G = \SL_2(\F_3)$, we have $Z(G) = 1$.
\end{lemma}

\begin{proof}
	The conjugation with $e_1$ induces an automorphism of order $|G' :G''|-1$ on $G'/G''$, and this action is irreducible. Any element in $Z(G)$ is fixed under the conjugation with $e_1$. If $Z(G)$ is nontrivial, this implies $|G': G''| \in \{4,8,9\}$ by \cite[Lemma]{ALP66}. For $|G' : G''| = 4$, we obtain $|H| = 3$ (see Lemma~\ref{lemma:structuredagl}) and $|G''| = 2$. With this, we easily deduce $G \cong \SL_2(\F_3)$. For $|G': G''| = 8$, we obtain $|H| = 7$ and $|G''| \in \{2,4\}$, which yields $|G| \in \{112,224\}$. 
	If $|G': G''| = 9$, we obtain $|H| = 8$ and $|G''| = 3$, so $|G| = 216$. A straightforward verification using GAP \cite{GAP4} shows that the groups of order $112$, $216$ and $224$ do not satisfy the assumptions of this section. 
\end{proof}

Until the end of Section~\ref{sec:specialcasetsimple}, we assume $G \neq \SL_2(\F_3)$, which implies $Z(G) = 1$ (see Example~\ref{ex:slcont} and Lemma~\ref{lemma:center}). 
By Remark~\ref{rem:theobabd}, we may fix $h_1 \in C_H(G'') \setminus \{1\}$ and a corresponding element $g_1  \in G' \setminus G''$ as defined in Lemma~\ref{lemma:realconj}. Set 
\begin{equation}\label{eq:c}
C  \coloneqq\left\{[a,g_1] \colon a \in G'\right\} \subseteq G''. 
\end{equation}

Since $G'$ has nilpotency class 2, $C$ is a subgroup of $G$. 
We collect some observations on the conjugacy classes in $G'$: 
\begin{lemma}\label{lemma:hcyclicintersectionddg}
	$\null$
	\begin{enumerate}[(i)]
		\item We have $C = \langle [e_1^m g_1 e_1^{-m},g_1] \colon m \in \Z \rangle.$ 
		\item For $u \in G''$, we have $[g_1 u ]_{G'} = C g_1 u$. 
		\item For any $x \in G' \setminus G''$, we have $[x]_{G'} = [x]_G \cap x G''$, which yields $[x]_G = \bigcup_{m = 0}^{s-1} e_1^m \cdot [x]_{G'} \cdot e_1^{-m}.$
	\end{enumerate}
\end{lemma}

\begin{proof}
	$\null$
	\begin{enumerate}[(i)]
		\item Since every element in $G'$ can be expressed as $e_1^m g_1 e_1^{-m} d$ with $m \in \Z$ and $d \in G'' = Z(G')$, we have $C = \left\{[a,g_1] \colon a \in G'\right\} =  \langle [e_1^m g_1 e_1^{-m}, g_1] \colon m \in \Z \rangle$.
		\item We have $[g_1 u]_{G'} = \{[a,g_1 u] \colon a \in G' \} g_1 u = \{[a,g_1] \colon a\in G'\} g_1 u= C g_1 u$ due to $u \in  Z(G')$. 
		\item We first prove $[x]_{G'} = [x]_G \cap x G''$. Clearly $[x]_{G'} \subseteq [x]_G \cap x G''$. Now let $k \in G$ with $k x k^{-1} \in x G''$ and write $k = e_1^\ell u$ with $\ell \in \Z$ and $u \in G'$. As $u x u^{-1} \in x G''$ holds, conjugation with $e_1^\ell$ fixes $x G''$. By Remark \ref{rem:eitrivial}, this yields $e_1^\ell = 1$ and hence $k \in G'$, which shows the equality. For the second part, note that the $G$-conjugates of $[x]_{G'}$ are contained in $[x]_G$. On the other hand, we have 
		\[[x]_G \subseteq G' \setminus G'' = \bigcup_{m = 0}^{s-1} e_1^m x e_1^{-m} G''\] 
		and for every $m \in \{0, \ldots, s-1\}$, conjugation with $e_1^m$ induces a bijection between $[x]_{G'} = [x]_G \cap xG''$ and $[x]_G \cap e_1^m x e_1^{-m} G''$. \qedhere
	\end{enumerate}
\end{proof}

For the proof of Theorem~\ref{theo:CaminaDsinglenormal}, we use of the following simplification:

\begin{Corollary}\label{cor:caminac}
	The group $G'$ is a Camina group if and only if $C = G''$. 
\end{Corollary}

\begin{proof}
	Note that $C = G''$ is equivalent to $[g_1]_{G'} = g_1 G''$. The claim now follows from the fact that conjugation with powers of $e_1$ permutes the nontrivial cosets of $G''$ in $G'$ transitively. 
\end{proof}
In the following, we assume that $C$ is a proper subgroup of~$G''$ and construct an element $y \in \soc(ZFG)$ which is not contained in $(G')^+ \cdot FG$. By Lemma~\ref{lemma:socideal}, this is a contradiction, and Theorem~\ref{theo:CaminaDsinglenormal} follows.

\begin{Remark}\label{const:cindp}\label{rem:homomorphismondg}\label{rem:prodconst}
	By Lemma~\ref{lemma:nilpotencyclassofdg}, $G''$ is elementary abelian. In particular, there exists a nontrivial group homomorphism $\alpha \colon G'' \to \F_p$ with $\alpha(C) = 0$. 
	For $g \in G$, we set 
	\[a_g \coloneqq \begin{cases}
	\alpha(u) &\text{ if } g \text{ is conjugate to $g_1u$ in $G$ for some } u \in G'', \\
	0 &\text{ otherwise.} 
	\end{cases}\]
This is well-defined: Let $g_1 u_1$ and $g_1 u_2$ with $u_1, u_2 \in G''$ be conjugate in $G$. By Lemma \ref{lemma:hcyclicintersectionddg}, we have 
$g_1 u_2 \in [g_1 u_1]_G \cap g_1 u_1 G'' = [g_1 u_1]_{G'} = C g_1 u_1 = g_1 u_1 C$,
which yields $u_2 \in u_1 C$ and hence $\alpha(u_1) = \alpha(u_2)$. Since $\alpha$ is a group homomorphism,
$a_{g_1 u_1 u_2} =  a_{g_1 u_1} + a_{g_1 u_2}$
for all $u_1, u_2 \in G''$. For $k \in \Z$, conjugation with $e_1^k$ yields
\[a_{e_1^k g_1 e_1^{-k} u_1 u_2} = a_{e_1^k g_1 e_1^{-k} u_1} + a_{e_1^k g_1 e_1^{-k} u_2}.\]

For $g \in G'',$ conjugation with $e_1$ fixes $t \coloneqq \prod_{g' \in [g]_G} g' \in G'' = Z(G')$, which implies $t \in Z(G) = 1$. By the above, we obtain
\begin{equation*}
	\sum_{g' \in [g]_G} a_{g_1 g'} = a_{g_1 \prod_{g' \in [g]_G} g'} = a_{g_1} = 0. 
\end{equation*}
\end{Remark}

For the remainder of this section, we consider the element $y \coloneqq \sum_{g \in G} a_g g \in FG$ with the coefficients given in Remark~\ref{const:cindp}. Clearly $y \in ZFG$ and $y \notin (G')^+ \cdot FG$ as $\alpha$ is nontrivial. Our aim is to show that $y \in \soc(ZFG)$. We begin with a preparatory observation:

\begin{lemma}\label{lemma:yannihilatess}
For every subgroup $U \subseteq G''$ with $|U| > 2$, we have $y \cdot U^+ = 0.$ 
\end{lemma}

\begin{proof}
The coefficient of $w \in G$ in the product $y \cdot U^+$ is given by $\sum_{u \in U} a_{wu^{-1}}$. All summands are zero unless $w \in G' \setminus G''$. Write $w = e_1^k g_1 e_1^{-k}d$ for some $k \in \Z$ and $d \in G''$. Then we have 
	\begin{equation*}
	\begin{aligned}
	\sum_{u\in U} a_{wu^{-1}} = \sum_{u \in U} a_{e_1^k g_1 e_1^{-k} du^{-1}} = |U| \cdot a_{e_1^k g_1 e_1^{-k}d} +  a_{e_1^k g_1 e_1^{-k} \prod_{u \in U}u^{-1}} = a_{e_1^k g_1 e_1^{-k}} =  \alpha(1) = 0.
	\end{aligned}
	\end{equation*}
	
	In the second equality, we used Remark~\ref{rem:homomorphismondg}. For the third, note that $p$ divides $|U|$ and that $\prod_{u \in U} u^{-1} = 1$ since $U$ is elementary abelian of order at least 3. 
\end{proof}

We now show in several steps that $y$ annihilates the basis of $J(ZFG)$ described in Lemma~\ref{lemma:structjzfg}. 

\begin{lemma}\label{lemma:yannihilatesbz}
	For $g \in G' \setminus \{1\}$, we obtain $y \cdot b_{[g]} = 0$. 
\end{lemma}

\begin{proof}
Set $[g] \coloneqq [g]_G$ and $\ell \coloneqq |[g]|$. First, we assume $g \in G''$.  The coefficient of $t \in G$ in the product $y \cdot b_{[g]}$ is given by
	\[\sum_{g' \in [g]} a_{tg'^{-1}} - \ell \cdot a_{t}.\] 
	Again, we assume $t \in G' \setminus G''$ since all summands are zero otherwise. Then $t$ is conjugate in $G$ to an element $g_1 d$ with $d \in G''$. Since $y \cdot b_{[g]} \in ZFG$ has constant coefficients on the conjugacy classes of $G$, we may assume $t = g_1 d$. 
Using Remark~\ref{rem:homomorphismondg}, we obtain
	\begin{equation*}
	\sum_{g' \in [g]} a_{tg'^{-1}}  = \sum_{g' \in [g]} a_{g_1 d g'^{-1}} \\
	= \ell \cdot a_{g_1 d} + \sum_{g' \in [g^{-1}]} a_{g_1 g'}= \ell \cdot a_{g_1 d} = \ell \cdot a_t.
	\end{equation*}
	
	Now let $g \in G'\setminus G''$. Then $g$ is $G$-conjugate to $g_1 d$ for some $d \in G''$ and we may assume $g = g_1 d.$  
	Note that $C$ is nontrivial due to $g_1 \notin Z(G')$. If $|C| = 2$, then all conjugacy classes in $G'$ are of length at most two (see Lemma~\ref{lemma:hcyclicintersectionddg}). By \cite[Proposition 3.1]{ISH99}, $G'$ is an extraspecial group, which implies $|G''| = 2 = |C|$. This is a contradiction to $C \subsetneq G''$, so $|C| > 2$. By Lemma \ref{lemma:hcyclicintersectionddg}, $[g]$ is a union of cosets of the subgroups $e_1^\ell C e_1^{-\ell}$ with $\ell \in \Z$. Lemma~\ref{lemma:yannihilatess} yields $y \cdot (e_1^\ell C e_1^{-\ell})^+ = 0$ and hence $y$ annihilates $b_{[g]} = [g]^+$.
\end{proof}

Now we show that $y$ annihilates the basis elements $b_C$ of $J(ZFG)$ corresponding to conjugacy classes $C$ which are not contained in $G'.$ We first consider the conjugacy classes of the nontrivial elements in $C_H(G'')$. 

\begin{lemma}\label{lemma:chddg}
	The element $y$ annihilates $[h]_G^+$ for every $h \in C_H(G'') \setminus \{1\}.$
\end{lemma}

\begin{proof}
	Recall that $[h] \coloneqq [h]_G = U_{G,h} \cdot h$ with $U_{G,h} = \{1\} \cup \left\{e_1^\ell g_1' e_1^{-\ell} \colon 0 \leq \ell \leq s-1\right\}$ for some $g_1' \in G' \setminus G''$ which is conjugate to $g_1$ in $G$ (see Lemma \ref{lemma:fiallconjugate}). By conjugating with a suitable power of $e_1$, we may choose $g_1' \in g_1 G''$. By Lemma~\ref{lemma:hcyclicintersectionddg}, we have $g_1' \in g_1 G'' \cap [g_1] = [g_1]_{G'} = g_1 C$, so we write $g_1' = g_1c$ for some $c \in C$. Note that $y \cdot [h]^+ = 0$ is equivalent to $y \cdot U_{G,h}^+ =0.$ Again, we need to show that
	\begin{equation}\label{eq:sumut}
	\sum_{u \in U_{G,h}} a_{tu^{-1}} = 0
	\end{equation}
 for all $t \in G$, and we may assume $t \in G'$ since all summands are zero otherwise. 
	\medskip
	
	First let $t \in G' \setminus G''$ and write $t = e_1^k g_1 d' e_1^{-k}$ with $k \in \Z$ and $d' \in G''$. Setting $d \coloneqq c^{-1}d' \in G''$ yields $t = e_1^k g_1' d e_1^{-k}$. We have
	\begin{equation}\label{eq:summe1}
	\sum_{u \in U_{G,h}} a_{tu^{-1}} = \sum_{u \in U_{G,h}} a_{e_1^{-k}tu^{-1} e_1^k} = \sum_{u \in U_{G,h}} a_{g_1'd e_1^{-k} u^{-1} e_1^k} = \sum_{u \in U_{G,h}} a_{g_1' u^{-1}d},
	\end{equation}
	since conjugation with $e_1^{-k}$ permutes the set $U_{G,h}$ and $d \in Z(G')$. Furthermore, we obtain
	\begin{alignat*}{1}
	\left\{g_1' u^{-1} d\colon u \in U_{G,h} \right\} &= \left\{g_1' (e_1^\ell g_1'^{-1} e_1^{-\ell}) d \colon 0 \leq \ell \leq s-1\right\} \cup \{g_1'd\}\\
	&= \left\{[g_1', e_1^\ell] d \colon 0 \leq \ell \leq s-1\right\} \cup \{g_1' d\} \eqqcolon X.
	\end{alignat*} 
	For $\ell \in \{1, \ldots, s-1\}$, we consider the element $x \coloneqq [g_1', e_1^\ell]d \in X$. Since $[g_1', e_1^\ell]$ is conjugate to $g_1'$ in $G$ by Lemma~\ref{lemma:realconj}, there exist $d_x \in C_{g_1'}$ and $m_x \in \{0, \ldots, s-1\}$ with $x = e_1^{m_x} g_1' d_x e_1^{-m_x} d$. For $x \coloneqq g_1'd \in X$, we obtain the same form by setting $m_x = 0$ and $d_x = 1$. Note that $C_{g_1'} = C$ because of $g_1' \in g_1 G''$ (see Lemma~\ref{lemma:hcyclicintersectionddg}). Since $U_{G,h}$ is a set of representatives for the cosets of $G''$ in $G'$, this is also true for $X$ since its elements arise from the elements in $U_{G,h}$ by inversion and multiplication with $g_1'd$. In particular, the elements in $X \setminus \{d\}$ form a system of representatives for the nontrivial cosets of $G''$ in $G'$. Hence the correspondence $x \leftrightarrow m_x$ is one-to-one. Setting $v_{m_x} \coloneqq d_x$ for $m_x \in \{0, \ldots, s-1\}$, this yields $x = e_1^{m_x} g_1' v_{m_x} e_1^{-m_x}d$ for all $x \in X \setminus \{d\}$ and hence 
	\[X = \left\{e_1^m g_1' v_m e_1^{-m} d \colon 0 \leq m \leq s-1\right\} \cup \{d\}.\]
	Using $a_d = 0$ and that the coefficients of $y$ are constant under conjugation with $e_1^{-m}$, we obtain the following for the sum given in \eqref{eq:summe1}:
	\[\sum_{u \in U_{G,h}} a_{tu^{-1}} = \sum_{x \in X} a_x = \sum_{m = 0}^{s-1} a_{e_1^m g_1' v_m e_1^{-m}d}  = \sum_{m = 0}^{s-1} a_{g_1' v_m e_1^{-m} d e_1^m} = \sum_{m = 0}^{s-1} a_{g_1 c v_m e_1^{-m} d e_1^m}.\] 
	In the last step, we used $g_1' = g_1 c$. We obtain
	\begin{equation}\label{eq:aux}
	\sum_{m = 0}^{s-1} a_{g_1 c v_m e_1^{-m} d e_1^m} = \sum_{m = 0}^{s-1} \left(a_{g_1 c v_m}  + a_{g_1 e_1^{-m} d e_1^m}\right) = \sum_{m = 0}^{s-1} a_{g_1 e_1^{-m} d e_1^m}= \frac{s}{|[d]|} \sum_{d' \in [d]} a_{g_1 d'} = 0
	\end{equation}
(for $[d] \coloneqq [d]_G$). 
	The first step is due to Remark~\ref{rem:homomorphismondg}. In the second, we use that $c v_m \in C$ for all $m \in \{0, \ldots, s-1\}$.
	The third equality follows since the element $e_1^{-m} d e_1^m$ traverses the conjugacy class $[d] = \{e_1^\ell d e_1^{-\ell} \colon 0 \leq \ell \leq |[d]|-1\}$ exactly $s/|[d]|$ times. In the last step, we apply Remark \ref{rem:prodconst} again.
	\medskip
	
	Now let $t \in G''$. By Lemma \ref{lemma:realconj}\,(iii), there exists $g \in G$ with $g_1'^{-1} = g g_1 g^{-1}$. Using $a_t = 0$ as well as $t \in Z(G')$, we obtain
	\[\sum_{u \in U_{G,h}} a_{tu^{-1}} = a_t + \sum_{i = 0}^{s-1} a_{t e_1^i g_1'^{-1} e_1^{-i}} = \sum_{i = 0}^{s-1} a_{t (e_1^i g) g_1 (e_1^i g)^{-1}} = \sum_{i = 0}^{s-1} a_{g_1 (e_1^i g)^{-1} t (e_1^i g)}.\]
	Due to $t \in Z(G')$, we may assume $g \in H$, so $e_1$ and $g$ commute. Setting $t' \coloneqq g^{-1} t g$, we have
	\[\sum_{i = 0}^{s-1} a_{g_1 (e_1^i g)^{-1} t (e_1^i g)} = \sum_{i = 0}^{s-1} a_{g_1 (g e_1^i )^{-1} t (g e_1^i)} = \sum_{i = 0}^{s-1} a_{g_1 e_1^{-i} t' e_1^i} = a_{g_1 \prod_{i = 0}^{s-1} e_1^{-i} t' e_1^i} = 0\] 
	(using Remark \ref{rem:homomorphismondg} and arguing analogously to the third step in \eqref{eq:aux}). Hence \eqref{eq:sumut} holds for every $t \in G$, which completes the proof.
\end{proof}

Now we gather our results to prove the following statement:

\begin{lemma}\label{lemma:yannihilateslengthp}
	For every $g \in G \setminus G',$ we have $y \cdot [g]_G^+ = 0.$ 
\end{lemma}

\begin{proof}
As before, we set $\bar{g} \coloneqq gG'' \in D$. Let $k \in \Z$ with $[g]_G \subseteq e_1^k G'$.
Since $D \cong \AGL(1,|D'|)$ by assumption, we have $[\bar{g}]_D = \bar{g} D'$.
In particular, $[g]_G$ contains an element~$e_1^k d$ with $d \in G''$ and we may assume $g = e_1^k d$. A set of representatives for the cosets of $C_G(e_1^k)$ in $G$ can be chosen in $G'$ since $e_1^k$ centralizes $H$. Similarly, a set of representatives for the cosets of $C_G(d)$ in $G$ can be found in $H = \langle e_1 \rangle$. Since $d$ and $e_1^k$ centralize $G'$ and $\langle e_1 \rangle$, respectively, we have $[e_1^k d]_G = [e_1^k]_G \cdot [d]_G$. Moreover, $G'$ acts on $[e_1^k d]_G$ by conjugation with orbits of the form $[e_1^k]_G d'$ with $d' \in [d]_G$. In particular, $[e_1^kd]_G^+$ is a multiple of $[e_1^k]_G^+$ in $FG$, so it suffices to prove $y \cdot [e_1^k]_G^+ = 0$. 
	\medskip
	
If $e_1^k \in C_G(G'')$, the claim follows by Lemma~\ref{lemma:chddg}. Now let $h \coloneqq e_1^k \notin C_G(G'')$. As $h$ acts on $D' \setminus \{1\}$ without fixed points (see Remark~\ref{rem:eitrivial}), we deduce that $[g, h] \in G''$ for some $g \in G'$ is equivalent to $g \in G''$. Moreover, $[a_1 a_2, h] = [a_1, h] \cdot [a_2, h]$ for all $a_1, a_2 \in G''$ since $G''$ centralizes $G'$. Hence $N \coloneqq U_{G,h} \cap G''$ is a normal subgroup of $G$. Since $h$ acts nontrivially on $G''$, we have $N \neq 1$. Now consider an orbit $B$ of the conjugation action of $G''$ on $[h]_G$ and let $b \in B$. We claim that $B = Nb$. To see this, write $b = ghg^{-1}$ for some $g \in G'$. For every $n \in N$, there exists some $d \in G''$ with $[d,h] = n$ by the above. Since $[d,b] = [d,ghg^{-1}] = [d,h] = n$ follows from the fact that $g$ centralizes $G''$, we have $nb = dbd^{-1} \in B$. On the other hand, we have $[d,h] \in N$ for $d \in G''$, which implies $B \subseteq Nb$. This shows that $[h]_G$, and hence $U_{G,h}$, is a disjoint union of cosets of~$N$. If $|N| = 2$, we have $N \subseteq Z(G) = 1$, which is a contradiction. Hence $|N|> 2$ and Lemma~\ref{lemma:yannihilatess} yields $y \cdot N^+ = 0$, which implies $y \cdot [h]_G^+ = 0$.
\end{proof}

With these preliminary results, we now prove Theorem \ref{theo:CaminaDsinglenormal}:

\begin{proof}[Proof of Theorem \ref{theo:CaminaDsinglenormal}]
Suppose that $G'$ is not a Camina group. In particular, this implies $G \neq \SL_2(\F_3)$ (see Example~\ref{ex:slcont}), so $Z(G) = 1$ (see Lemma~\ref{lemma:center}). By Corollary~\ref{cor:caminac}, the group $C$ defined in~\eqref{eq:c} is a proper subgroup of $G''$. Consider the element $y = \sum_{g \in G} a_g g$ with the coefficients described in Remark~\ref{const:cindp}. Clearly, we have $y \in ZFG$. By Lemmas~\ref{lemma:yannihilatesbz} and~\ref{lemma:yannihilateslengthp}, $y$ annihilates the basis of $J(ZFG)$ given in Lemma~\ref{lemma:structjzfg}, which implies $y \in \soc(ZFG)$. By Lemma~\ref{lemma:socideal}, this is a contradiction to $\soc(ZFG) \trianglelefteq FG$ due to $y \notin (G')^+ \cdot FG$. Hence $G'$ is a Camina group. 
\end{proof}

\subsection{Converse implication}\label{sec:proofc}
As a preparation for the proof that (ii) implies (i) in Theorem~\ref{theo:b}, we prove the following result:

\begin{theorem}\label{theo:rueck}
Let $G$ be a finite basic group with $Z(G') = G''$ such that the following hold:
%
%
	
	\begin{enumerate}[(i)]
		\item $G/G'' \cong \AGL(1,|G'/G''|)$, 
		\item $G'$ is a Camina group, and
		\item $C_G(G'') \not \subseteq G'$.
	\end{enumerate}
Then $\soc(ZFG)$ is an ideal in $FG$. 
\end{theorem}

\begin{proof}
As before, we write $G = G' \rtimes H$ for a Hall $p'$-subgroup $H$, $D = G/G''$ and set $\bar{g} \coloneqq g G'' \in D$ for every $g \in G$. Due to $D \cong \AGL(1, |D'|)$, $H$ is cyclic and for every $h \in H \setminus \{1\}$, we have $[\bar{h}]_D = \bar{h}D'$. Due to $C_G(G'') \not \subseteq G'' = Z(G')$, there exists a non-trivial element $r \in C_H(G'')$. Since $r$ centralizes $G''$, we have $|[r]_G| = |[\bar{r}]_D| = |D'|$ (see \cite[Lemma 3.30]{BRE222}). Note that the elements in $D' \setminus \{1\}$ are $D$-conjugate by (i). With this, one can show analogously to the proof of Lemma \ref{lemma:realconj} that the elements in $U_{G,r} \setminus \{1\}$ are conjugate in~$G$.
\medskip

By \cite[Remark 3.15]{BRE222}, we have $\soc(ZFG) = \bigoplus_{h \in H} S_h$ for $S_h \coloneqq \soc(ZFG) \cap FhG'$. As in Example~\ref{ex:sl}, it suffices to show $\dim_F S_h = 1$ for all $h \in H$. First consider an element $y = \sum_{g \in G} a_g g \in S_1$. 
Then $y \cdot [r]_G^+ = 0$ and hence $y \cdot U_{G,r}^+ = 0$. The coefficient of $1$ in the latter product is given by $a_1 + \sum_{u' \in U_{G,r} \setminus \{1\}}a_{u'^{-1}}$. Since all elements in $U_{G,r} \setminus \{1\}$ are conjugate in $G$, we obtain 
\[0 = a_1 + \sum_{u' \in U_{G,r} \setminus \{1\}}a_{u'^{-1}} = a_1 + (|U_{G,r}|-1) a_{u^{-1}} = a_1 + (|D'|-1) a_{u^{-1}}\] 
for every $u \in U_{G,r}$, which implies $a_1 = a_{u^{-1}}$. Now let $g \in G' \setminus G''$. Then we find $u \in U_{G,r}$ with $u \in gG''$. By~(ii), $g$ and $u$ are conjugate in $G$ and hence $a_g = a_u = a_1$ follows. For $g \in G'' \setminus \{1\}$, considering the coefficient of $1$ in the product $y \cdot b_{[g]} = 0$ as above yields $a_{g^{-1}} = a_1$. 
This shows $y \in F(G')^+$ and hence $\dim_F S_1 = 1$. Now let $y = \sum_{g \in G} a_g g \in S_h$ for $h \in H \setminus \{1\}$. Consider a conjugacy class $K \in \Cl(G)$ with $K \subseteq h G'$. Due to $[\bar{h}]_D = \bar{h}D'$, the class $K$ contains an element $h t$ with $t \in G''$. Let $[t] \coloneqq [t]_G$. If $t \neq 1$, the coefficient of~$h$ in the product $y \cdot b_{[t^{-1}]} = 0$ is given by 
\[0 = \sum_{t' \in [t]} a_{ht'} - |[t]| \cdot a_h.\] 
Note that $[t] = \{btb^{-1} \colon b \in H\}$ due to $t \in Z(G')$. Since $H$ is abelian, this yields $h[t] \subseteq [ht] = K$, so the summands $a_{ht'}$ for $t' \in [t]$ are all equal. Again, this implies $a_h = a_{ht}$. This shows $y \in F(hG')^+$ and hence $\dim_F S_h = 1$ follows. As in Example~\ref{ex:sl}, we conclude that $\soc(ZFG)$ is an ideal in $FG$.
\end{proof}

\begin{Example}\label{ex:extraspecial}
$\null$
\begin{enumerate}[(i)]	
\item For $p > 2$, let $P$ be the extraspecial group of order $p^{3}$ and exponent $p$. Then there exists an automorphism $\alpha \in \Aut(P)$ of order $p^2-1$ that permutes the nontrivial cosets of $P'$ in $P$ transitively (see~\cite[Theorem]{WIN72}). The group $G \coloneqq P \rtimes \langle \alpha \rangle$ satisfies the assumptions of Theorem~\ref{theo:rueck}. In particular, $\soc(ZFG)$ is an ideal in $FG$. 
\medskip

For $p = 3$, this construction yields the group $ \texttt{SmallGroup}(216,86)$ in GAP, which was studied in \cite[Example~3.26]{BRE222}. For $p = 2$, an analogous construction extending the extraspecial group $Q_8$ by a cyclic group of order 3 leads to the group $\SL_2(\F_3)$ studied in Example~\ref{ex:sl}. 
\item Let $p \in  \mathbb{P}$ be a  prime number and let $q$ be a power of $p$ such that $q+1$ is not divisible by $3$. Let~$P$ be a Sylow $p$-subgroup of $\operatorname{PSU}(3,q^2)$ and let $G$ be its normalizer in $\operatorname{PSU}(3,q^2)$. By \cite[Theorem~II.10.12]{HUP67}, $G = P \rtimes H$ with a cyclic group $H$ of order $q^2-1$. Moreover, $|P| = q^3$ and $P' = Z(P)$ is cyclic of order $q$. In particular, $P$ is a Camina group. Moreover, it is easily verified that $G/P' \cong \AGL(1,|P/P'|)$ and that $C_G(P') \not \subseteq P'$.  
Hence $G$ satisfies the assumptions of Theorem~\ref{theo:rueck} and thus $\soc(ZFG) \trianglelefteq FG$. In view of (i), we point out that in this case $P$ is not necessarily extraspecial. 
\end{enumerate}
\end{Example}

\subsection{Proof of Theorem~\ref{theo:b}}\label{sec:prooftheob}

With the results of the preceding subsections, we now prove Theorem~\ref{theo:b}:

\begin{proof}[Proof of Theorem~\ref{theo:b}]
Let $G$ be a finite basic group with $Z(G') = G''$. 
\medskip

First assume $\soc(ZFG) \trianglelefteq FG$. Then $G = G_1 * \dots * G_n$ for the subgroups $G_1, \ldots, G_n$ defined in Theorem~\ref{theo:qi}. Applying Remark~\ref{rem:theobabd} and Theorem~\ref{theo:CaminaDsinglenormal} to each of the groups $G_1, \ldots, G_n$ yields the statement of (ii) with $K_i = G_i'$.
\medskip

Conversely, assume (ii) and consider a factor $G_i = K_i \rtimes \langle \varphi_i \rangle$ of the central product. The assumptions (a) and (b) imply $G_i' = K_i$ and $G_i/G_i'' \cong \AGL(1,|K_i/K_i'|) =  \AGL(1,|G_i/G_i''|)$. Applying Theorem~\ref{theo:rueck} to~$G_i$ yields $\soc(ZFG_i) \trianglelefteq FG_i$. Then $\soc(ZFG) \trianglelefteq FG$ follows by Lemma~\ref{lemma:centralproduct}.
\end{proof}

For $p>2$, we obtain the following simplification of Theorem~\ref{theo:b}:
 
\begin{Remark}
Assume $p > 2$ and let $M \coloneqq K \rtimes \langle \varphi \rangle$ be a finite group with $K$ and $\varphi$ satisfying the properties (a) and (b) of Theorem~\ref{theo:b}\,(ii). Let $\alpha$ denote the unique involution in $\langle \varphi \rangle$. Then $\alpha$ acts on the abelian group $K/K'$ by inversion. For $k_1, k_2 \in K \setminus K'$, write $\alpha(k_i) = k_i^{-1} k_i'$ for some $k_i' \in K'$ ($i = 1,2$). Then $\alpha([k_1,k_2]) = [k_1^{-1} k_1', k_2^{-1}k_2'] = [k_1^{-1}, k_2^{-1}] = [k_1,k_2]$ since $K' = Z(K)$ by property (a). In particular, $\alpha$ acts trivially on $K'$ and hence $\langle \varphi \rangle$ does not act faithfully on $K'$.
\end{Remark} 

This shows that condition (c) in Theorem~\ref{theo:b}\,(ii) can be omitted if $p$ is odd.
 
\subsection*{Acknowledgment}
The results in this paper are part of my PhD thesis \cite{BRE22}, which was supervised by Burkhard K\"ulshammer. I would like to thank him for his advice as well as numerous valuable remarks concerning this project. I am also grateful to Lászlo Héthelyi and Magdolna Szöke for providing me with Example~\ref{ex:extraspecial}\,(ii). Moreover, I would like to thank Pascal Schweitzer for helpful comments on the introduction of this paper. Finally, I thank the anonymous reviewer for their detailed feedback.

\subsection*{Declaration of interests}

The author reports that there are no competing interests to declare.

\bibliographystyle{plain}
\bibliography{finitegroups.bib}

\end{document}